\documentclass[11pt,thmsa]{article}%
\usepackage{amsmath}
\usepackage{amsfonts}
\usepackage{amssymb}
\usepackage{graphicx}%
\newtheorem{theorem}{Theorem}[section]

\newtheorem{corollary}[theorem]{Corollary}

\newtheorem{definition}[theorem]{Definition}

\newtheorem{lemma}[theorem]{Lemma}

\newtheorem{proposition}[theorem]{Proposition}

\newenvironment{proof}[1][Proof]{\noindent\textbf{#1.} }{\ \rule{0.5em}{0.5em}}

\begin{document}

\title{Clifford-Wolf Translations of Finsler spaces\footnote{Supported by NSFC (no, 10671096, 10971104) and SRFDP of China}}
\author{  Shaoqiang Deng$^{1,}$ and Ming Xu$^2$
\thanks{Corresponding author. E-mail: dengsq@nankai.edu.cn}\\
$^1$School of Mathematical Sciences and LPMC\\
Nankai University\\
Tianjin 300071, P. R. China\\
$^2$Department of Mathematical Sciences\\
Tsinghua University\\
Beijing 100084, P. R. China}
\date{}
\maketitle
\begin{abstract}
In this paper, we study Clifford-Wolf translations of Finsler
spaces. We first give a characterization of Clifford-Wolf
translations of Finsler spaces in terms of Killing vector fields. In
particular, we show that there is a natural correspondence between
Clifford-Wolf translations and the Killing vector fields of constant
length. In the special case of homogeneous Randers spaces, we give
some explicit sufficient and necessary conditions for an isometry to
be a Clifford-Wolf translation. Finally, we construct some explicit
examples to explain some of the results of this paper.


\end{abstract}
\section{Introduction}
Let $(X, d)$ be a locally compact connected metric space. A
Clifford-Wolf translation of $(X, d)$ is an isometry $\rho$ of
$(X,d)$ onto itself such that the function $d(x, \rho (x))$ is
constant on $X$. Geometrically, a Clifford-Wolf translation is an
isometry  of the metric space such that all the points in $X$ move
the same distance. In the Riemannian case, it is an important
problem to determine all the Clifford-Wolf translation of an
explicit Riemannian manifold. For example, in the non-positive
curvature case, J. A. Wolf proved that an isometry of the Riemannian
manifold is a Clifford-Wolf translation if and only if it is bounded
(\cite{WO64}). Moreover, if a Riemannian manifold $(M, Q)$ has
strictly negative curvature, then any Clifford-Wolf translation of
$(M, Q)$ must be trivial. The above consequences have an important
corollary that a homogeneous Riemannian manifold with non-positive
curvature admits a transitive solvable Lie group of isometries. This
result is the basis for many works on homogeneous Riemannian
manifolds with negative (non-positive) curvature, see for example
\cite{HE74, AW76}. It should be noted that J. A. Wolf's results have
been generalized to  general Finsler spaces, see \cite{DP}. Using
these results, we can hopefully get a classification of negatively
curved homogeneous Finsler spaces.

The purpose of this article is to initiate the  study of
Clifford-Wolf translations of Finsler spaces. Let $(M, F)$ be a
Finsler space, where $F$ is positively homogeneous of degree one but
not necessary absolutely homogeneous. Then we can define the
distance function $d$ of $(M, F)$. Although generically $d$ is not
reversible, we can define a Clifford-Wolf translation of $(M, F)$ to
be an isometry $\rho$ of $(M, F)$ onto itself such that $d(x, \rho
(x))$ is a constant on $M$. It is an interesting problem to study
the problem that to what extent the results on Clifford-Wolf
translations of Riemannian manifolds still hold for Finsler spaces.
In particular, it is hopeful that some new phenomena will happen in
this more general case.

In the paper we first give a characterization of the Clifford-Wolf translations of Finsler spaces, in terms of Killing vector fields. In particular, we show that there is a natural correspondence between Wolf-Clifford translations and the Killing vector fields of constant length. In the special case of homogeneous Randers spaces, we give some explicit sufficient and necessary conditions for an isometry to be a clifford-Wolf translation. Finally, we construct some explicit examples to explain some of the results of this paper.

The arrangement of the paper is as the following: In Section 2, we present some fundamental knowledge on Finsler geometry, in particular the definition of the Riemann curvature of Finsler metrics. In Section 3, we study Clifford-Wolf translations of general Finsler spaces and obtain a characterization of Clifford-Wolf translations in terms of Killing vector fields of constant length. In Section 4, we apply the results to homogeneous Randers spaces and obtain some explicit description of Clifford-Wolf translations. Finally, in Section 5, we construct some examples in some special compact Lie groups to explain some of the results.

\section{Preliminaries}
\begin{definition}
A Finsler metric on a manifold $M$ is a function $F:TM\rightarrow \mathbb{R}^+$, which is smooth on the slit tangent bundle $TM\backslash \{0\}$. In any
local coordinates $(x^i,y^j)$ for $TM$, with $x^i$'s the local coordinates for $M$ and
$y^j$'s the local coordinates for the tangent vectors, $F$ satisfies the following conditions:
\begin{description}
\item{(1)}\quad $F(x,y)>0$ for any $y\neq 0$.
\item{(2)}\quad $F(x,\lambda y)=\lambda F(x,y)$ for any $y\in TM_x$ and $\lambda >0$.
\item{(3)}\quad The Hessian matrix defined by $g_{ij}=\frac{1}{2}[F^2]_{y^i y^j}$ is positive definite.
\end{description}
\end{definition}

A Finsler metric $F$ is called reversible if it satisfies $F(x,y)=F(x,-y)$.

A Finsler metric of the form $F=\alpha+\beta$, where $\alpha$ is a Riemannian metric and $\beta$ is a 1-form,
 is called a Randers metric. In this case, the positive definiteness of the metric is equivalent to the condition that the length of  $\beta$ with respect to the metric $\alpha$ is everywhere smaller than $1$. A Randers metric is reversible if and only if it is Riemannian.

As a generalization of Riemannian geometry, almost all concepts in Riemannian geometry have
their counterpart in Finsler geometry. The connections and curvatures have already
been discussed in a lot of works, see for example \cite{BCS00} and \cite{CS04}. Below we only briefly recall the notion of Riemann curvature of a Finsler space. This will be useful in the following sections.

Let $(M, F)$ be a connected Finsler space and $V$ a nowhere zero vector field on a
  open subset $U$ of $M$. Then we can define an affine connection on the tangent bundle $TU$ over $U$, denoted by $\nabla^V$,
  such that the following holds:
  \begin{description}
  \item{(a)}\quad $\nabla^V$ is torsion-free:
  $$\nabla^V_X Y-\nabla^V_YX=[X,Y],$$
   for all vector fields $ X, Y$ on $U$,
  \item{(b)}\quad $\nabla^V$ is almost metric-compatible:
  $$Xg_V(Y, Z)=g_V(\nabla^V_XY, Z)+g_V(Y, \nabla^V_XZ)+2C_V(\nabla^V_XV, Y, Z),$$
   for all vector fields $ X, Y,Z$ on $U$.
    \end{description}
In the above formulas, $g_V$, resp. $C_V$,  is the fundamental tensor, resp.  Cartan tensor, of $(M, F)$. It is worthwhile to point out that  the conditions (a), (b) can be used to deduce an explicit formula for the Chern connection, see \cite{RA} for the details.

A nowhere zero vector field $V$ on $U$ is called a geodesic vector field if $\nabla^V_VV=0$. This is equivalent to the condition that the flow lines of $V$ are all geodesics. In this case, $V$ is also a geodesic vector field with respect to the Riemannian metric $g_V$. The vector field $V$  is called parallel if $\nabla^V_XV=0$, for any vector field $X$ on $U$.

The above  affine connection has close relation to the Chern connection presented in \cite{BCS00} and \cite{CS04}. In particular, using the properties
of the Chern connection we have
the following conclusion: if $V, W$ are two nowhere zero vector fields on $U$ and $V(x)=W(x)$, then
$$(\nabla^V_XY)(p)=(\nabla^W_XY),$$
for any vector fields $X, Y$ on $U$. This fact enable us to define the covariant derivative  along any curve. Let
$c(t)$ be a curve on $U$ and $V, X$ be  vector fields along $c$, with $V$ nowhere zero. Extending $V$, $X$ and $c'$ to a subset
containing $c$, we can define $(\frac{\nabla^V}{dt} X)(t)=(\nabla_{c'}^V X)(t)$. The above fact shows that
this is independent of the extensions. In case $V$ and $c'$ coincide, we denote the derivative simply by $\frac{\nabla}{dt}$.

Now we define the Riemann tensor of $(M, F)$. Let $X$ be a non-zero tangent vector of $M$ at $x$ and $c_X$ be the geodesic starting from $x$ with initial vector $X$. Given $Y\in T_x(M)$, we can construct a geodesic variation with variation vector field $Y(s,t)$ such that
$Y(0,0)=Y$. We then define
$$R^X(Y)=\frac{\nabla^2}{dt}Y(s, 0)|_{s=0}.$$
It can be checked that the right hand of the above equation does not depend on the choice of the variation. It is called  the Riemann curvature operator  of $(M, F)$. See Shen's book \cite{SH2} for more details.

The following result is very important, see \cite{SH2} for a proof.
\begin{lemma}
 {\rm (Shen \cite{SH2})}\quad For any nonzero vector $X$ in $T_x(M)$ with a nonzero geodesic field $V$ extending $X$ in an open neighborhood $U$ of $x$, the Riemann curvature operator $R^X$ of the Finsler space coincides with the Jacobi operator $\bar{R}^X$ of the osculating Riemannian  metric $g_V$.
 \end{lemma}

Note that $V$ is nowhere zero on $U$, so that the Riemannian metric $g_V$ is well-defined on $U$. Also recall that
the Jacobi operator $\bar{R}^X$ of $g_V$ is defined by
$$\bar{R}^X(Y)=R(Y, X) X,\quad Y\in T_x(M),$$
where $R$ is the curvature tensor of $g_V$.

Now we generalize the notions of   Clifford-Wolf translations and Killing vector fields  to Finsler
geometry.
\begin{definition}
A Clifford Wolf translation $\rho$ is an isometry of the Finsler manifold $(M,F)$ which
satisfies the condition that the distance from any $x\in M$ to its image $\rho(x)$ is a constant.
\end{definition}
It should be be noted that, unlike in  the Riemannian case, the reverse $\rho^{-1}$ of
a Clifford Wolf translation $\rho$ of a Finsler space may not be a Clifford Wolf translation, for a Finsler metric  may not be
reversible.

In \cite{DH02}, it is proved the group of isometries of a Finsler space   $(M,F)$ is a Lie transformation group on $M$. Its Lie algebra is
the space of Killing vector fields. Equivalently, a smooth vector field on $(M,F)$ is a Killing vector field if and only if the flow $\phi_t$  generated by $X$  are
isometries of $(M,F)$.

\section{Clifford-Wolf translations and Killing vector fields of constant length}

 Clifford-Wolf translations and Killing vector fields of constant length
are closely related. In the Riemannian case, if a Killing vector field $X$ generates a family of
Clifford-Wolf translations $\varphi_t$, where $t$ is close enough to $0$, then $X$ must have
 constant length. If we assume the Riemannian manifold to be compact,
 then any Clifford Wolf translation close to
the identity map (in the topology of the Lie group of isometries) is generated by a Killing vector field of constant length (see \cite{BN081, BN082, BN09}). We now show that  these
statements are also correct in the Finsler case.

\begin{lemma}\label{b}
Let $X$ be a Killing vector field on a Finsler space $(M,F)$ and $U\subset M$ be a open subset such that $X$ is nowhere zero on $U$.  Then on $U$ we have $\nabla^X_X
X=-\frac{1}{2}\tilde{\nabla}^{(X)}|X|^2$, where $\tilde{\nabla}^{(X)}
f=g_{ij}(X)f_{x^j}\frac{\partial}{\partial x^i}$ is the gradient
field of $f$ for the Riemannian metric $g_X=g_{ij}(X)$ on $V$.
\end{lemma}
\begin{proof}
 Choose a local coordinate system $x=(x^i)$ with
$X=\frac{\partial}{\partial x^1}$, which implies that the flow $\phi_t$
defined by $X$ is just a shift of $x^1$ by $t$ in this chart. The assumption that $X$ is a
Killing vector field for $F$ means that $F(x,y)$ is independent of $x^1$. By
definition,
\begin{eqnarray}
\nabla^X_X X &=& 2G^i(X)\frac{\partial}{\partial x^i} \nonumber \\
      &=& \frac{1}{2}g^{il} ([F^2]_{x^m y^l}y^m - [F^2]_{x^l})
      \nonumber \\
      &=& \frac{1}{2}g^{il}(X)([F^2]_{x^m y^l}y^m)(X)\frac{\partial}{\partial x^i} -
      \frac{1}{2}\tilde{\nabla}^{(X)}|X|^2.
\end{eqnarray}
Consider the value  at $X$. Since $y^2=\cdots=y^n=0$, and $y^1=1$, we have
\begin{equation}
\frac{1}{2}g^{il}(X)[F^2]_{x^m
y^l}y^m(X)=\frac{1}{2}g^{il}(X)[F^2]_{x^1 y^l}=0,
\end{equation}
since by the assumption $F^2$ is independent of $x^1$. This completes the proof.
\end{proof}

As an immediate consequence, we have
\begin{corollary}
If $X$ is a Killing vector field on $(M,F)$, then all flow curves of $X$ are
geodesics if and only if $X$ has constant length.
\end{corollary}

The relations between the Clifford-Wolf translations and Killing vector fields for a Finsler manifold
are stated in the following two theorems.

\begin{theorem}
Suppose the complete Finsler manifold $(M,F)$ has a
positive injective radius. If $X$ is a Killing
field on $(M,F)$ of constant length and $\varphi_t$ is the flow generated
by $X$, then $\varphi_t$  is a Clifford Wolf translation for all sufficiently small
$t>0$.
\end{theorem}
\begin{proof}
 Denote the injective radius of $(M, F)$  by $r>0$  and suppose the constant
length $F(X)$ of $X$ is $l>0$. For any $t\in(0,r/l)$, the geodesic
flow curve from $x$ to $\phi_{t}(x)$ is contained in the geodesic
ball $B_r(x)$. Then the distance from $x$ to $\phi_t(x)$ is the
length of flow curve from $x$ to $\phi_t(x)$, which is $tl$. From
this the theorem follows.
\end{proof}

\begin{theorem}
Let $(M,F)$ be a compact Finsler space. Then there is a
$\delta>0$, such that any Clifford-Wolf translation $\rho$ with
$d(x,\rho(x))<\delta$ is generated by a Killing vector field of constant
length.
\end{theorem}
\begin{proof}
Denote by $G$ the group of isometries of $(M,F)$ and by $\mathfrak g$ the Lie algebra of $G$.
Then $G$ is a compact Lie group (\cite{DH02}). For sufficiently
small $\delta>0$, any
Clifford-Wolf translation $\rho$ with $d(x,\rho(x))<\delta$ is contained in
a neighborhood $V\subset G$ of the identity map which can be generated by the exponential map. In particular,
there is a
Killing vector field $X$ with flow $\phi_t$ such that $\phi_1=\rho$. By the compactness,
we can assume that $\delta$ is smaller than the injective radius, i.e. the distant from $x$ to
$\phi_t(x)$ for each $t\in [0,1]$ is the length of the unique geodesic within the geodesic
ball $B_{\delta}(x)$.
Now we prove that  $X$ has a constant length $l$. Then it follows from \ref{b} immediately that
the distance from $x$ to $\phi_1(x)$ is the constant $tl$.
It is obvious that $F(X)$ is constant along each flow curve. By \ref{b},
$\nabla^X_X X=0$ along any flow curve where $F(X)$ takes
its minimum or maximum.
In either case, we have $d(x,\rho(x))=F(X)$. If $\rho$ is a Clifford-Wolf translation,
then the minimum and maximum of $F(X)$ will be equal. Hence $X$ is a Killing vector field of constant length.
\end{proof}

Below is another direct corollary of Lemma \ref{b}.
\begin{corollary}
If $X$ is a Killing vector field on $(M,F)$ of constant length, then the
covariant derivative $\nabla^X_X R^X=0$, where
$R^X:TM\rightarrow TM$ is the Riemann curvature of $F$.
\end{corollary}
\begin{proof}
 By \ref{b}, $X$ is a geodesic vector field. So neither the connection coefficients for
$D_X$ nor the curvature tensor $R^X$ is changed when they are
evaluated at $X$ and $F$ is replaced by the Riemannian metric $g_X$.
Hence $\nabla^X_X R^X$ remains the same when we change $F$ to $g_X$.
As $F(X)=g_X(X,X)$, the vector field $X$ has the same constant
length for $g_X$.
 From the proof of Lemma \ref{b}, it is easily seen that $X$ is also a Killing vector field for $g_X$.
On the  special coordinates chosen there, $F$ is independent of
$x^1$, so  $g_{ij}(X)=\frac{1}{2}[F^2]_{y^i y^j}$ is also
independent of $x^1$. From \cite{BN09}, we know that a Riemannian
Killing vector field of constant length satisfies $\tilde{\nabla}_X
R^X = 0$. Therefore the equality remains valid for $F$.
\end{proof}

In the Riemannian case, this corollary is a key step to study Clifford-Wolf homogeneous spaces.
But in the Finsler case, it is much less useful.

For a Randers space $(M,F)$, a Finsler Killing vector field $X$ for $F$ is in fact also
a Riemannian Killing vector field for $\alpha$. In fact we have
\begin{lemma}
A vector field $X$ on a Randers space $(M,F)$, $F=\alpha+\beta$ is a Killing vector field if and only if
$X$ is a Killing vector field for $\alpha$ and $L_X \beta =0$.
\end{lemma}
\begin{proof}
If the vector field $X$ is a Killing vector field for $(M,F)$,
then its flow $\phi_t$ fixes $F$, i.e.
\begin{equation}
\phi_t^* F=\phi_t^* \alpha +\phi_t^* \beta=\alpha+\beta=F,
\end{equation}
for each $t$. For any $x$ and $y\in T_x(M)$, $y\ne 0$, considering the values of $F$ on $(x,y)$ and $(x,-y)$, we get that
 $\phi_t^* \alpha=\alpha$ for each $t$,
i.e., $X$ is Killing vector field for $\alpha$, and $\phi_t^*\beta=\beta$ for each $t$, i.e.,
$L_X \beta=0$. The other direction is obvious.
\end{proof}

When $X$ is a Killing vector field for $\alpha$, the condition $L_X \beta =0$ can be equivalently
written as $L_X V=[X,V]=0$, where $V$ is the dual of $\beta$ with respect to the Riemannian metric $\alpha$.

\section{Homogeneous Randers spaces and Clifford-Wolf translations}

A Randers space $(M,F)$, with $F=\alpha+\beta$, is homogeneous if
its full isometry group $G=\mbox{I}(M,F)$ acts transitively on $M$.
The Lie algebra of $G$ is denoted by $\mathfrak{g}$. The homogeneous
space $M$ can be identified with $G/H$, where $H$ is the isotropy
subgroup of $G$ at a point $x$ of $M$. Recall that $H$ is necessary
compact. The Randers metric $F$ is determined by the data at the
tangent space $T_x\cong \mathfrak{g}/\mathfrak{h}=\mathfrak{m}$,
i.e. $\alpha$ is determined by an inner product $<\cdot,\cdot>$ on
$\mathfrak{m}$, and the dual $V$ of $\beta$ is a vector of
$\mathfrak{m}$, both invariant under the Ad-action of $H$.

The presentations of homogeneous spaces for $(M,F)$ may not be
unique. The full isometry group can be changed to any closed
subgroup which acts transitively on the manifold. The problem of
finding Clifford-Wolf translations can be discussed in two different
style. We may try to find all Clifford-Wolf translations for
$(M,F)$, which means we may need to deal with different form of
homogeneous spaces for the same manifold, or take $G$ to be the
biggest one, the isometry group itself. Our discussion in this work
will take another style. Restrict our discussion to a possibly
smaller Lie group $G$ and its Lie algebra $\mathfrak{g}$, it is
relatively easy to start and already able to tell us something new
about Clifford-Wolf translations in Finsler geometry.

If we assume further that $M$ is compact, then $G$ is also compact.
There is an orthogonal decomposition for
the Lie algebra $\mathfrak{g}=\mathfrak{h}+\mathfrak{m}$ with respect to  the Killing form of $\mathfrak{g}$.
This is true since  $H\cap C(G)$ contains only the identity map of $M$. Though the Killing
form may not be definite, its vanishing space $C(\mathfrak{g})$ is totally contained in
$\mathfrak{m}$. For any $X\in \mathfrak{g}$, its decomposition will be denoted $X=X_{\mathfrak h }+ X_{\mathfrak m}$.

The study of Clifford-Wolf translations on a homogeneous Randers space is very interesting.
From  Section 3, we have seen that, the problem of finding Clifford-Wolf translations, at least
those which are close to the identity map, can be reduced to the problem of finding Killing
fields of constant length.

We can calculate the differential of the length function $F(X)$ of a Killing vector field $X$ when it
is viewed as an element of $\mathfrak{g}$. The vector of $X$ at each
point can be pulled back to the chosen point, by some $g\in G$, and can be identified with
$(\mbox{Ad}_g X)_m$ in $\mathfrak{m}$. The evaluation of $F$ at that point is then
\begin{equation}
F=<(\mbox{Ad}_g X)_m, (\mbox{Ad}_g X)_m>^{1/2}+<(\mbox{Ad}_g X)_m, V>.
\end{equation}
for a family $g_t=\exp(tY)\cdot g$, with $Y\in \mathfrak{g}$,
the derivative for $t$ at $t=0$ is
\begin{equation}
\frac{d}{dt}F|_{t=0}=\frac{<[Y,\mbox{Ad}_g X]_m, (\mbox{Ad}_g X)_m>}{<(\mbox{Ad}_g X)_m,(\mbox{Ad}_g X)_m>^{1/2}}+<[Y,\mbox{Ad}_g X]_m,V>.
\end{equation}
From this we get    the following theorem.
\begin{theorem}\label{main theorem}
Let $M=G/H$ be a connected homogeneous Randers space with  $G=\mbox{I}(M,F)$ and $F=\alpha+\beta$.
Then any Killing vector field $X$ of constant length satisfies
\begin{equation}
\frac{<[Y,\mbox{Ad}_g X]_m, (\mbox{Ad}_g X)_m>}{<(\mbox{Ad}_g X)_m,(\mbox{Ad}_g X)_m>^{1/2}}+<[Y,\mbox{Ad}_g X]_m,V>=0, \label{Clifford-Wolf-hrs-1}
\end{equation}
for all $Y\in \mathfrak{g}$, and $g\in G$.
If $H$ is connected, then (\ref{Clifford-Wolf-hrs-1}) is also a sufficient condition for a Killing vector field $X$ to have constant length.
\end{theorem}

If we choose the family $g_t=g \cdot exp(tY)$ to calculate the derivative, then (\ref{Clifford-Wolf-hrs-1})
is changed to
\begin{equation}
\frac{<(\mbox{Ad}_g [Y,X])_m, (\mbox{Ad}_g X)_m>}{<(\mbox{Ad}_g X)_m,(\mbox{Ad}_g X)_m>^{1/2}}+<(\mbox{Ad}_g [Y,X])_m,V>=0,\label{Clifford-Wolf-hrs-2}
\end{equation}
for all $Y\in \mathfrak{g}$ and $g\in G$. In fact for
(\ref{Clifford-Wolf-hrs-2}) to be sufficient, we only need to take
all $Y\in\mathfrak{m}$ and $g\in G$. And it is similar for
(\ref{Clifford-Wolf-hrs-1}).

This theorem  is very useful for finding Killing vector fields
of constant length and Clifford-Wolf
translations of homogeneous Randers spaces.

For example, we have the following theorem.

\begin{theorem}\label{easy theorem}
Suppose $G$ is a connected compact Lie group endowed with a Randers metric $F=\alpha+\beta$, such that
$\alpha$ is bi-invariant. Denote  the dual of $\beta$ with respect to  $\alpha$ by $V$.
Then the following four conditions are equivalent:
\begin{description}
\item{(1)}\quad  The vector $X\in \mathfrak{g}$ generates  a Killing vector field of constant length.
\item{(2)}\quad The ideal of ${\mathfrak g}$ generated by $[\mathfrak{g},X]$ and the ideal generated by $V$ are orthogonal
to each other with respect to $\alpha$.
\item{(3)}\quad The ideal generated by $[\mathfrak{g},V]$ and the ideal generated by $X$ are orthogonal
to each other with respect to $\alpha$.
\item{(4)}\quad With respect to the metric $\alpha$, the inner product between any vectors of the $\mbox{Ad}_G$-orbits $\mathcal{O}_X$ and
$\mathcal{O}_V$ is a constant.
\end{description}
\end{theorem}
\begin{proof} We will prove this theorem by showing that (i) implies (i+1), for $i=1,2,3$ and that (4) implies (1).

That (1) implies  (2).\quad
If $M$ is a connected compact Lie group and $\alpha$ is bi-invariant, then any $X\in \mathfrak{g}$
has  constant length with respect  to $\alpha$. The equation
(\ref{Clifford-Wolf-hrs-2}) can then be simplified as
\begin{equation}
<\mbox{Ad}_g [Y, X],V>=0,
\end{equation}
for all $Y\in \mathfrak{g}$ and $g\in G$. So for any $g=\exp t_1 Z_1 \exp t_2 Z_2\cdots \exp t_k Z_k$
and $g'=\exp t'_1 Z'_1 \exp t'_2 Z'_2 \cdots \exp t'_l Z'_l$,
\begin{equation}
<\mbox{Ad}_g [Y,X], \mbox{Ad}_{g'} V>=0.
\end{equation}
Differentiating the above equation with respect to all $t_i$ and $t'_j$ and evaluating at the $0$'s, we get
\begin{equation}
<[Z_1,\ldots,[Z_k,[Y,X]]],[Z'_1,\ldots,[Z'_l,V]]>=0.
\end{equation}
This proves that (1) implies (2).

That (2) implies (3).\quad This is obvious since $<[Y,X],V>=-<X,[Y,V]>$. This also proves that  (3) implies (2).

That (3) implies (4).\quad  Suppose  $X'=\exp(ad_Y) X\in
\mathcal{O}_X$ and $V'=\exp(ad_Z) V\in \mathcal{O}_V$. Then using
(2), (3) and a direct calculation one easily shows that
$<X',V'>=<X,V>$.

That (4) implies (1).\quad If $<\mathcal{O}_X,\mathcal{O}_V>$ is a constant, then $\mathcal{O}$ is orthogonal
to tangent space of $\mathcal{O}_V$ at $V$, i.e., $[\mathfrak{g},V]$. So for any
$Y\in \mathfrak{g}$ and $g\in G$,
\begin{equation}
<[Y,\mbox{Ad}_g X],V>=-<\mbox{Ad}_g X,[Y,V]>=0.
\end{equation}
Using Theorem \ref{main theorem}, we finish the proof.
\end{proof}

Theorem \ref{easy theorem} explains the most obvious technique to construct a homogeneous
Randers metric on Lie groups and to find Clifford-Wolf translations. Take $G=G_1\times G_2$ with a
bi-invariant Riemannian metric $\alpha$, so that
their Lie algebras $\mathfrak{g}_1$ and $\mathfrak{g}_2$ have positive dimensions and
orthogonal. Then take $X$ from one factor and $V$ from the other. The Killing vector field
generated by $X$ have constant length for both $F$ and $\alpha$.
Then $X$ generates a Clifford-Wolf translation which is represented as left multiplication by group
elements.

In fact we can also mix it with right multiplications. Take $X_1$ from $\mathfrak{g}_1$
and $X_2$ from $\mathfrak{g}_2$, such that $[X_2,V]=0$. The composition of the left
multiplication generated by $X_1$ and right multiplication generated by $X_2$ generates
Clifford-Wolf translations for $F$. The corresponding Killing vector field is not an element of $\mathfrak{g}$,
provided $X_2$ is not in the center. Theorem \ref{main theorem} can be used to generalize
Theorem \ref{easy theorem} when the Killing vector field $X$ is not in $\mathfrak{g}$.


It will be more interesting to see the case
that $\alpha$ is not bi-invariant, and the Killing vector field is not of
a constant length with respect to $\alpha$, but has a constant length with respect to $F$.
For  convenience, in this case
the homogeneous Randers metric will usually be presented as
\begin{equation}
F(X)=\alpha(X)+<X,V>_{bi},
\end{equation}
where the $\beta$-term is expressed by the inner product $<\cdot,\cdot>_{bi}$
for a fixed bi-invariant metric, which is not the inner product of $\alpha$ (note that $\alpha$  is assumed to be not bi-invariant).

Let $G=G_1\times G_2$  with $G_1=SU(2)$
and $G_2=S^1$. Fix a bi-invariant linear metric on $G$ such  that $\mathfrak{g}_1$ and
$\mathfrak{g}_2$ are orthogonal. Then
$\mathfrak{g}_1$ can be orthogonally decomposed as
$\mathbb{R}U_1\oplus \mathbb{R}U_1^{\perp}$, where $U_1$ is a unit vector on ${\mathfrak g}_1$.
Let $U_2$ be a unit vector in $\mathfrak{g}_2$. Define a left invariant Riemannian metric $\alpha$  on $G$  such that
$\alpha|_{{\mathbb R} U_1^\perp}$ is proportional to
  $\alpha'$.
Choose $V=aU_1+bU_2$, i.e.
\begin{equation}
\beta(X)=<X,V>_{bi}=ax_1+bx_2,
\end{equation}
for $X=x_1 U_1 +X' +x_2 U_2$, $X'\in\mathbb{R}U_1^\perp$.

If the vector $X=rU_1+sU_3$ with $r$ and $s\neq 0$
generates a Killing vector field with  constant length for $F$,
then its orbit $\mathcal{O}_X$, which is a radius $r$ sphere
for the invariant metric centered at $sU_3$, has
the constant length $l$ for $F$. A sufficient condition for this is
\begin{equation}
\alpha^2(X')=(ax_1+bs-l)^2,\label{a}
\end{equation}
for all $X'=x_1 U_1 +X'+s U_2\in\mathcal{O}_X$ with
$x_1^2+<X',X'>_{bi}=r^2$. As the right side only depends on
$<X',X'>_{bi}^{1/2}$, we get $<X',U_2>=0$ and $<X',U_1>=0$ for
$\alpha$. So
\begin{eqnarray}
\alpha(X')^2 &=& a_{11}x_1^2+a_{22}<X',X'>_{bi}+a_{33}s^2+2a_{13}x_1 s\nonumber\\
             &=& (a_{11}-a_{22})x_1^2+2a_{13}s x_1 + (a_{33}s^2+ a_{22}r^2)
\end{eqnarray}
If it equals
\begin{equation}
(<X',V>_{bi}-l)^2=(ax_1+bs-l)^2,
\end{equation}
then we get a system of  linear equations for $a_{11}$, $a_{22}$, $a_{33}$ and $a_{13}$.
The space of the solutions of the above equations is of   dimension one.
We now summarize the above discussion  as a proposition which is useful in the discussion for
other compact Lie groups.
\begin{proposition}
Let $X$ and $V$ belong to the Lie algebra
$\mathfrak{g}=su(2)\oplus\mathbb{R}$ for the Lie group $SU(2)\times
S^1$. Assume $X$ lies neither   in  $su(2)$ nor in $\mathbb{R}$.
Then  for any chosen $l>0$, there is a one-dimensional family of
Randers metrics under which $X$ generates a Killing vector field of
constant length $l$.
\end{proposition}

The above construction still works even if $SU(2)$ is changed to a larger connected
compact Lie group $G_1$.
The unit sphere for a bi-invariant metric of $\mathfrak{g}_1$,
with a constant non-vanishing
$g_2$ component will have a constant length for some suitably chosen $F$.
Since the sphere contains more than one orbit, for any unit vector  $U_1$,
 the orbit of $X=r U_1+s U_3$, with $r$ and $s$ non-zero, contains more than one point.
This means that $\mathcal{O}_X$ is not of constant
length for $\alpha$.

In this construction,
we have only used the product structure in the Lie algebra.
Considering the quotient group of  $G$  by a finite subgroup in its center,  we have
the following
\begin{corollary}
If the connected compact Lie group $G$ has a center with positive
dimension, then there is a left invariant Randers metric
$F=\alpha+\beta$ on $G$, such that there are CLIFFORD-WOLF
translations generated by Killing fields in $\mathfrak{g}$, which
are not CLIFFORD-WOLF translations for $\alpha$.
\end{corollary}

\section{Clifford-Wolf translations for left invariant Randers metrics on $SU(n)$}

To find examples of Clifford-Wolf translations generated by
Killing vector fields of constant length
with respect to a  left invariant Randers metric
on a connected compact simple or semi-simple Lie group, it is natural to start
with $SU(2)$.
But we can easily see that on $SU(2)$, a left invariant Randers netric has no nontrivial Clifford-Wolf translations  unless it is a  Riemannian metric.
\begin{proposition}\label{prop} Any left invariant Randers metric on $SU(2)$
which has a nonzero Killing vector field of  constant length in
$su(2)$ is Riemannian and bi-invariant.
\end{proposition}
\begin{proof}
 Suppose there is a non-vanishing $X\in su(2)$ which generates a Killing vector field of
constant length for $F=\alpha+\beta$. Then
any one dimensional subalgebra of $su(2)$, which is in fact a Cartan subalgebra,
 contains exactly two vectors in the orbit
of $X$. These two vectors are opposite to each other but have the same $F$ length and
$\alpha$ length. So the restriction of $\beta$  to each one dimensional subalgebra must vanish.
Thus $\beta=0$. Hence $F$ is Riemannian. Further, the above argument also implies that the orbit $\mathcal{O}_X$ has the same length for $\alpha$.
Hence $\alpha$ is bi-invariant.
\end{proof}

On larger simple compact Lie groups, there may exist Clifford-Wolf translations
generated by a Killing vector field of constant length for a left invariant Randers metric which are not of constant length with respect to the underlying
Riemannian metric. Here we will only discuss the examples of $SU(3)$.
But the techniques used below can be applied to other classical compact Lie groups.

Let $X=\sqrt{-1}\mbox{diag}(-1,-1,2)$ and
$V=\sqrt{-1}\mbox{diag}(a,b,c)$ be two matrices in the Lie algebra
${\mathfrak su}(3)$. Choose the bi-invariant metric on $SU(3)$ whose
restriction to
 the Lie algebra is $<P,Q>_{bi}=\mbox{tr}P^* Q$.
We will show how to find a Riemannian metric $\alpha$ such
that for the Randers metric $F(\cdot)=\alpha(\cdot)+<\cdot,V>_{bi}$, the Killing vector field
generated by $X$ has length $l$.

In fact our choice of $X$ is very restrictive, and almost unique in some sense. First, $X$ must have exactly two eigenvalues, i.e., it is invariant under the
action of a
$\mathbb{Z}_2$ subgroup in the Weyl group. Otherwise the intersection of the orbit $\mathcal{O}_X$
with any Cartan subalgebra, which is $2$ dimensional and contains $6$ points. But any ellipse passing
the six points must be a circle centered at the origin. This means the Randers metric
$F$ has to be Riemannian, i.e.,  it has a  vanishing $\beta$,
when it is restricted to any Cartan subalgebra. Then $F$ is Riemannian following a similar
argument as for $SU(2)$. This means that, up to a non-zero scalar, $X$ can be chosen to have the
diagonal form $\sqrt{-1}\mbox{diag}(-1,-1,2)$. Obviously we can assume that $X$ and $V$ lies in the
 same Cartan subalgebra, so that $V$ has a diagonal form at
the same time.

The following proposition shows that, if we can find the required Riemannian metric $\alpha$, then it must be totally determined by $X$, $V$ and $l$.
\begin{proposition}
For any vectors $X$ and $V$ in ${\mathfrak su}(3)$, and a pre-chosen $l>0$, there is at most one
Randers metric $F(y)=\alpha(y)+<y,V>_{bi}$, $y\in {\mathfrak su}(3)$, such that the
Killing vector field generated by $X$ has constant length $l$.
\end{proposition}
\begin{proof}
 We only need  to consider $X=\sqrt{-1}\mbox{diag}(-1,-1,2)$.
If there are two $\alpha_1$ and $\alpha_2$ with corresponding $F_1$ and $F_2$,
such that $F_1=F_2$ on $\mathcal{O}_X$, then $\alpha_1=\alpha_2$ on $\mathcal{O}_X$.
We only need  to prove that $\alpha_1$ and $\alpha_2$ are equal on each Cartan subalgebra.
Each Cartan subalgebra of ${\mathfrak su}(3)$ contains $3$ points of $\mathcal{O}_X$. If we denote
the three points of $\mathcal{O}_X$ by $X_1$, $X_2$ and $X_3$, then the centralizer
$\mathfrak{Z}(X_1)\subset {\mathfrak su}(3)$ is isomorphic to ${\mathfrak su}(2)\oplus\mathbb{R}$, with its group
$G'\cong S(S^1\times U(2))\subset SU(3)$. The intersection of $O_X$ with $\mathfrak{Z}(X_1)$
contains $X'$ and the $Ad_{G'}$ orbit $\mathcal{O}'_{X_2}$, which is a $2$ dimensional
sphere passing both
$X_2$ and $X_3$. Restricted to $\mathfrak{Z}(X_1)$, $\alpha_1-\alpha_2$ vanishes on the cone
generated by $\mathcal{O}'(X_2)$ and the point $X_1$ outside the cone. So $\alpha_1\equiv\alpha_2$
on $\mathfrak{C}(X_1)$, which contains the Cartan subalgebra we consider.
\end{proof}

If $V=\sqrt{-1}\mbox{diag}(a,b,c)\in {\mathfrak su}(3)$ also has two distinct eigenvalues,
 we can assume
$V=\lambda\sqrt{-1}\mbox{diag}(-1,-1,2)$ by a suitable conjugation.
Notice that both $X$ and $V$ are invariant under the $Ad$-action of $G'=S(U(2)\times S^1)$.
From the uniqueness result of  Proposition \ref{prop}, the linear metric $\alpha$ on ${\mathfrak su}(3)$
must also be $Ad_{G'}$-invariant.
So it must have the form
\begin{equation}
\alpha(A)=x\mbox{tr}Q^2+y u^* u + z q^2,
\end{equation}
with $x,y,z\in\mathbb{R}$,
where
$$A=\sqrt{-1}\left(
            \begin{array}{cc}
              Q & u \\
              u^* & q \\
            \end{array}
          \right)
\in {\mathfrak su}(3),$$ with $Q$  Hermitian symmetric and $q=-\mbox{tr}Q$.

For any $X'=UXU^*$ in the orbit $\mathcal{O}_X$, with $$U=\left(
                                                           \begin{array}{cc}
                                                             U' & v_2 \\
                                                             v_1^* & u \\
                                                           \end{array}
                                                         \right),$$
$\alpha(X')^2$ can be expressed as a polynomial of $t=|u|^2$, i.e.
\begin{equation}
\alpha(X')^2=(9x-9y+9z)t^2+(-12x+9y-6z)t+(5x+z).
\end{equation}
By the assumption it is  equal to
\begin{equation}
(<X',V>_{bi}-l)^2=[9\lambda t-(l+3\lambda)]^2,
\end{equation}
Since the Killing vector field has constant length $l$ with respect
to $F$. By comparing the coefficients of $t$, all the parameters
$x$, $y$ and $z$ can be determined by $l$ and $\lambda$. As long as
we require $\lambda$ to be nonzero and close to $0$, the Killing
vector field generated by $X$ produces a  Clifford-Wolf translation
with respect to $F$ but not with respect to $\alpha$.

In general, even if we do not require $V=\sqrt{-1}\mbox{diag}(a,b,c)\in {\mathfrak su}(3)$ to have two distinct eigenvalues,
$X$ and $V$ are still invariant under the $Ad$-action of a maximal torus in $SU(3)$ consisting of the
diagonal matrices. Then  the metric $\alpha$ must have the form
\begin{equation}
\alpha(A)=x_1 a_{11}^2+x_2 a_{22}^2 + x_3 a_{33}^2 + y_1 |u|^2 +y_2 |v|^2 + y_3 |w|^2,
\end{equation}
where
\begin{equation}
A=\sqrt{-1}\left(
    \begin{array}{ccc}
      a_{11} & u & v \\
      \bar{u} & a_{22} & w \\
      \bar{v} & \bar{w} & a_{33} \\
    \end{array}
  \right)\in su(3).
\end{equation}
For any
\begin{eqnarray}
X' &=& \sqrt{-1}\left(
         \begin{array}{ccc}
           u       & v & 0 \\
           -\bar{v} & \bar{u} & 0 \\
           0 & 0 & 1 \\
         \end{array}
       \right)
       \left(
         \begin{array}{ccc}
           -1 & 0 & 0 \\
           0 & s & w \\
           0 & \bar{w} & 1-s \\
         \end{array}
       \right)
       \left(
         \begin{array}{ccc}
           \bar{u} & -v & 0 \\
           \bar{v} & u & 0 \\
           0 & 0 & 1 \\
         \end{array}
       \right)\nonumber \\
    &=&\left(
         \begin{array}{ccc}
           s|v|^2-|u|^2 & (s+1)uv & v w \\
           (s+1)\bar{u}\bar{v} & s|u|^2-|v|^2 & \bar{u}w \\
           \bar{v}\bar{w} & u\bar{w} & 1-s \\
         \end{array}
       \right)\in \mathcal{O}_X,
\end{eqnarray}
with $|u|^2+|v|^2=1$, $s\in\mathbb{R}$ and $|w|^2=2+s-s^2$,
$\alpha(X')^2$ can be expressed as a polynomial of $s$ and $t=|u|^2$.
Compare it with
\begin{equation}
(<X',V>_{bi}-l)^2=[a(s(1-t)-t)+b(st-1+t)+c(1-s)-l]^2,
\end{equation}
in which $c=-a-b$, we get a system of linear equations, from which
we can uniquely solve all the coefficients in $\alpha$, for suitable
choices of $l$, $a$ and $b$.

To conclude, we see the technique used here can prove a more general
proposition.

\begin{proposition}
On $SU(n)$ with $n>2$, there are non-Riemannian left invariant
Randers metric for which we can find nonzero Killing fields of
constant length from $su(n)$.
\end{proposition}

{\bf Acknowledgement. }

\end{document}